\documentclass[12pt]{amsart}

\usepackage{amsfonts,amsmath,amssymb,amsthm,amscd,enumerate}
\usepackage[all,cmtip]{xy}
\usepackage[titletoc, title]{appendix}
\usepackage[stable]{footmisc}

\newtheorem{thm}{Theorem}[section]
\newtheorem{lem}[thm]{Lemma}
\newtheorem{cor}[thm]{Corollary}
\newtheorem{prop}[thm]{Proposition}

\theoremstyle{definition}
\newtheorem{defn}[thm]{Definition}
\newtheorem{ex}[thm]{Example}
\newtheorem{rem}[thm]{Remark}

\begin{document}
\setlength{\pdfpagewidth}{8.5in}

\setlength{\pdfpageheight}{11in}

\newcommand         {\rar}[1]       {\stackrel{#1}{\longrightarrow}}
\def    \onto   {\twoheadrightarrow}
\newcommand         {\har}[1]       {\stackrel{#1}{\hookrightarrow}}
\newcommand{\lx}[2]{{\vphantom{#2}}^{#1}{#2}}
\newcommand{\code}[1]{\overline{C_{#1}}}
\newcommand\rank{\operatorname{rank}}
\newcommand\Fields{\operatorname{Fields}}
\newcommand\Sets{\operatorname{Sets}}
\newcommand\Supp{\operatorname{Supp}}
\newcommand\Spec{\operatorname{Spec}}
\newcommand\trdeg{\operatorname{trdeg}}
\newcommand\Hom{\operatorname{Hom}}
\newcommand\PHS{\operatorname{PHS}}
\newcommand\BS{\operatorname{BS}}
\newcommand\TORS{\operatorname{TORS}}
\newcommand\mdim{\operatorname{mdim}}
\newcommand\ord{\operatorname{ord}}
\newcommand\Id{\operatorname{Id}}
\newcommand\Tr{\operatorname{Tr}}
\newcommand\Prof{\operatorname{Prof}}
\newcommand\Gen{\operatorname{Gen}}
\newcommand\Rep{\operatorname{Rep}}
\newcommand\Sch{\operatorname{Sch}}
\newcommand\Spl{\operatorname{Spl}}
\newcommand\SB{\operatorname{SB}}
\newcommand\A{\operatorname{\mathcal{A}}}
\newcommand\Func{\operatorname{\mathcal{F}}}
\newcommand\FuncG{\operatorname{\mathcal{G}}}
\newcommand\FuncT{\operatorname{\mathcal{T}}}
\newcommand\FuncD{\operatorname{\mathcal{D}}}
\newcommand\F{\operatorname{\mathbb{F}}}
\newcommand\G{\operatorname{\mathbb{G}}}
\newcommand\bbG{\operatorname{\mathbb{G}}}
\newcommand\HH{\operatorname{\mathbb{H}}}
\newcommand\M{\operatorname{\mathbb{M}}}
\newcommand\N{\operatorname{\mathbb{N}}}
\newcommand\Z{\operatorname{\mathbb{Z}}}
\newcommand\bbZ{\operatorname{\mathbb{Z}}}
\newcommand\bbF{\operatorname{\mathbb{F}}}
\newcommand\Char{\operatorname{char}}
\newcommand\bbN{\operatorname{\mathbb{N}}}
\newcommand\C{\operatorname{\mathbb{C}}}
\newcommand\Q{\operatorname{\mathbb{Q}}}
\newcommand\Proj{\operatorname{\mathbb{P}}}
\newcommand\bbP{\operatorname{\mathbb{P}}}
\newcommand\bG{\overline{G}}
\newcommand\bH{\overline{H}}
\newcommand\bGG{\operatorname{\overline{\mathfrak{G}}}}
\newcommand\CatFk{\textrm{\bf{Fields}}_k}
\newcommand\CatS{\textrm{\bf{Sets}}}
\newcommand\ds{\displaystyle}
\newcommand{\mb}[1]{\mathbb{#1}}

\def \F  {\mathbb{F}}
\def    \ind    {\mathrm{ind}}
\def    \tpsi   {\tilde{\psi}}

\newcommand\GL{\operatorname{GL}}
\newcommand\PGL{\operatorname{PGL}}
\newcommand\SL{\operatorname{SL}}
\newcommand\SO{\operatorname{SO}}
\newcommand\Spin{\operatorname{Spin}}
\newcommand\OO{\operatorname{O}}
\newcommand\GO{\operatorname{GO}}
\newcommand\PGO{\operatorname{PGO}}
\newcommand\SP{\operatorname{SP}}
\newcommand\GSP{\operatorname{GSP}}
\newcommand\PGSp{\operatorname{PGSp}}
\newcommand\E{\operatorname{E}}
\newcommand\ed{\operatorname{ed}}
\newcommand\p{\operatorname{p}}
\newcommand\cdim{\operatorname{cdim}}
\newcommand\cd{\operatorname{cdim}}
\newcommand\wt{\operatorname{wt}}
\newcommand\pw{\operatorname{w}}
\newcommand\Mat{\operatorname{M}}
\newcommand\Gal{\operatorname{Gal}}
\newcommand\spann{\operatorname{span}}
\newcommand\End{\operatorname{End}}
\newcommand\Aut{\operatorname{Aut}}
\newcommand\Core{\operatorname{Core}}
\newcommand\Stab{\operatorname{Stab}}
\newcommand\Ker{\operatorname{Ker}}
\newcommand\im{\operatorname{im}}
\newcommand\Code{\operatorname{Code}}
\newcommand\Br{\operatorname{Br}}
\newcommand     \ra     {\rightarrow}
\newcommand     \xra     {\xrightarrow}
\newcommand     \inj    {\hookrightarrow}
\newcommand     \surj    {\twoheadrightarrow}

\title{Essential dimension and error-correcting codes}

\author{Shane Cernele} 
\address{Department of Mathematics, University of British Columbia, 
Vancouver, BC, Canada, V6T 1Z2} 
\email{scernele@math.ubc.ca, reichst@math.ubc.ca, athena@math.ubc.ca} 
\thanks{This paper is based on a portion of the first author's Ph.D.~thesis 
completed at the University of British Columbia.
Both authors gratefully acknowledge financial support
from the University of British Columbia and the
Natural Sciences and Engineering Research Council of Canada.} 

\author{Zinovy Reichstein, {\tiny with an appendix by Athena Nguyen}}

%

\dedicatory{To the memory of Robert Steinberg}

\begin{abstract} 
One of the important open problems in the theory of central simple
algebras is to compute the essential dimension of $\GL_n/\mu_m$, 
i.e., the essential dimension of a generic division algebra of 
degree $n$ and exponent dividing $m$.
In this paper we study the essential dimension of groups  
of the form
\[ G=(\GL_{n_1} \times \dots \times \GL_{n_r})/C \, , \]
where $C$ is a central subgroup of 
$\GL_{n_1} \times \dots \times \GL_{n_r}$. Equivalently, we 
are interested in the essential dimension of a generic $r$-tuple 
$(A_1, \dots, A_r)$ of central simple algebras 
such that $\deg(A_i) = n_i$ and
the Brauer classes of $A_1, \dots, A_r$ satisfy a system 
of homogeneous linear equations in the Brauer group. 
The equations depend on the choice of $C$ via the error-correcting code
$\Code(C)$ which we naturally associate to $C$.
We focus on the case where $n_1, \dots, n_r$ 
are powers of the same prime.
The upper and lower bounds on $\ed(G)$ we obtain
are expressed in terms of coding-theoretic parameters of $\Code(C)$, 
such as its weight distribution.
Surprisingly, for many groups of the above form
the essential dimension becomes easier to estimate 
when $r \geq 3$; in some cases we even compute the exact value. 
The Appendix by Athena Nguyen contains an explicit description of the
Galois cohomology of groups of the form
$(\GL_{n_1} \times \dots \times \GL_{n_r})/C$. This
description and its corollaries are used throughout the paper.
\end{abstract} 

\subjclass[2010]{Primary 20G15, 16K20, 16K50. Secondary 94B05}
   
\maketitle 

\section{Introduction}\label{s:intro}
Let $k$ be a base field.  Unless otherwise specified, we will assume that 
every field appearing in this paper contains $k$ and every homomorphism 
(i.e., inclusion) of fields restricts to the identity map on $k$. 

We begin by recalling the definition of essential dimension of
a covariant functor $\mathcal{F}$ from the category of fields
to the category of sets. Given a field $K$ and an object 
$\alpha \in \mathcal{F}(K)$, we will say that
$\alpha$ {\em descends} to an intermediate field $k \subset K_0 \subset K$
if $\alpha$ lies in the image of the natural map $\mathcal{F}(K_0) \to
\mathcal{F}(K)$. The essential dimension $\ed(\alpha)$ of $\alpha$ is 
defined as the minimal value of $\trdeg_k(K_0)$ such that
$\alpha$ descends to a subfield $k \subset K_0 \subset K$.  Given a prime 
integer $p$, the essential dimension $\ed_p(\alpha)$ of $\alpha$ at $p$
is defined as the minimal value of $\trdeg_k(K_0)$, where the minumum is
taken over all finite field extensions $L/K$ and all intermediate
intermediate fields $k \subset K_0 \subset L$, such that 
$[L:K]$ is prime to $p$ and $\alpha_L$ descends to $K_0$.  

The essential dimension $\ed(\mathcal{F})$ (respectively, the
essential dimension $\ed_p(\mathcal{F})$ at $p$)
of the functor $\mathcal{F}$ is defined
as the maximal value of $\ed(\alpha)$ (respectively of $\ed_p(\alpha)$), 
where the maximum is taken over all feld extensions $K/k$ 
and all objects $\alpha \in \mathcal{F}(K)$.

Informally speaking, $\ed(\alpha)$ is the minimal number of independent
parameters required to define $\alpha$, $\ed(\mathcal{F})$ is the minimal 
number of independent parameters required to define any object 
in $\mathcal{F}$, and $\ed_p(\alpha)$, $\ed_p(\mathcal{F})$ 
are relative versions of these notions at a prime $p$. These relative 
versions are somewhat less intuitive, but they tend to be more accessible 
and more amenable to computation than
$\ed(\alpha)$ and $\ed(\mathcal{F})$. Clearly $\ed(\alpha) \geqslant
\ed_p(\alpha)$ for each $\alpha$, and $\ed(\mathcal{F}) \geqslant 
\ed_p(\mathcal{F})$.  
In most cases of interest $\ed(\alpha)$
is finite for every $\alpha$. On the other hand, $\ed(\mathcal{F})$ 
(and even $\ed_p(\mathcal{F})$) can be infinite.
For an introduction to the theory of essential dimension, we refer 
the reader to the 
surveys~\cite{BF03},~\cite{icm}, \cite{whatis} and~\cite{merkurjev-survey}.

To every algebraic group $G$ one can associate the functor 
\[ \text{$\mathcal{F}_G := H^1(\ast, G)$: $K \mapsto
\{$isomorphism classes of $G$-torsors over $\Spec(K) \}$.} \]
If $G$ is affine, then the essential dimension of this functor is 
known to be finite; it is usually denoted by $\ed(G)$, rather than
$\ed(\mathcal{F}_G)$. For many specific groups $G$, $H^1(K, G)$
is in a natural bijective correspondence with the set of isomorphism 
classes of some algebraic objects defined over $K$. In such cases, 
$\ed(G)$ may be viewed as the minimal number of independent 
parameters required to define any object of this type. 
This number is often related to classical problems in algebra.
 
For example, in the case where $G$ is the projective linear group $\PGL_n$,
the objects in question are central simple algebras.  That is,
\begin{equation} \label{e.gc1}
\text{$H^1(K, \PGL_n) = \{$isomorphism classes of cenral 
simple $K$-algebras of degree $n\}$}.  
\end{equation}
The problem of computing $\ed(\PGL_n)$ is one of the important open 
problems in the theory of central simple algebras; 
see~\cite[Section 6]{abgv}. 
This problem was first posed by C.~Procesi, 
who showed (using different terminology) that 
\begin{equation} \label{e.procesi}
\ed(\PGL_n) \leqslant n^2 \, ;
\end{equation}
see~\cite[Theorem 2.1]{procesi}. Stronger (but still quadratic)
upper bounds can be found in~\cite[Theorem 1.1]{lrrs} 
and~\cite[Theorem 1.6]{lemire}.
  
A more general but closely related problem is computing $\ed(\GL_n/\mu_m)$, 
where $m$ and $n$ are positive integers and $m$ divides $n$. Note that 
\begin{equation} \label{e.gc2}
\begin{array}{rr} 
H^1(K, \GL_n/\mu_m) = & \text{$\{$isomorphism classes of central simple 
$K$-algebras} \\ 
 & \text{of degree $n$ and exponent dividing $m \}$}.
\end{array} 
\end{equation}
In particular, $\ed(\PGL_n) = \ed(\GL_n/\mu_n)$. 
The problem of computing $\ed(\GL_n/\mu_m)$ partially 
reduces to the case where $m = p^s$ and $n = p^a$ are powers 
of the same prime $p$ and $1 \leqslant s \leqslant a$.

{\em From now on we will always assume that} $\Char(k) \neq p$.
The inequalities 
\begin{eqnarray} \label{e.bm}
\quad
\quad
p^{2a-2} + p^{a-s} \geqslant
\ed_p(\GL_{p^a}/{\mu_{p^s}}) & \geqslant
& \begin{cases}(a-1)2^{a-1} & \textrm{if $p=2$ and $s=1$,}
\\(a-1)p^a+p^{a-s} & \textrm{otherwise,} \end{cases}
\end{eqnarray}
proved in~\cite{BM12} represent a striking improvement on the best
previously known bounds. (Here $a \geqslant 2$.) Yet the gap between 
the lower and upper bounds in~\eqref{e.bm}
remains wide. The gap between the best known upper and lower bounds
becomes even wider when $\ed_p(\GL_{p^a}/\mu_{p^s})$
is replaced by $\ed(\GL_{p^a}/\mu_{p^s})$.

These gaps in our understanding of $\ed(\GL_n/\mu_m)$
will not deter us from considering the vastly more general problem 
of computing the essential dimension of groups of the form 
\begin{equation} \label{e.G}
G := (\GL_{n_1} \times \dots \times \GL_{n_r})/C,  
\end{equation}
in the present paper. Here $n_1, \dots, n_r \geqslant 2$ are integers, 
and $C \subset \bbG_m^r$ is a central subgroup of 
$\GL_{n_1} \times \dots \times \GL_{n_r}$.

As usual, we will identify elements $(m_1, \dots, m_r)$ of 
$\bbZ^r$
with characters $x \colon \bbG_m^r \to \bbG_m$, where 
$x \colon (\tau_1, \dots, \tau_r) \to
\tau_1^{m_1} \dots \tau_r^{m_r}$. The subgroup $C \subset \bbG_m^r$
is completely determined by the $\bbZ$-module
\begin{equation} \label{e.integer-code}
X(\bbG_m^r/C) = \{ (m_1, \dots, m_r) \in \bbZ^r \;  
| \; \tau_1^{m_1} \dots \tau_r^{m_r} = 1  \; \; 
\forall (\tau_1, \dots, \tau_r) \in C \} \, , 
\end{equation}
consisting of characters of $\bbG_m^r$ which vanish on $C$.
The Galois cohomology of $G$ is explicitly described in the appendix:
by Theorem~\ref{thm.appendix},
$H^1(K, G)$ is naturally isomorphic to the set of isomorphism classes
of $r$-tuples $(A_1, \dots, A_r)$ of central simple $K$-algebras 
such that 
\[ \text{$\deg(A_i) = n_i$} \quad \text{and} \quad 
\text{$A_1^{\otimes m_1} \otimes \dots \otimes A_r^{\otimes m_r}$ 
is split over $K$} \]
for every $(m_1, \dots, m_r) \in X(\bbG_m^r/C)$. 
(Note that in the special case where $r =1$, we 
recover~\eqref{e.gc1} and~\eqref{e.gc2}.)  It follows from this description
that the essential dimension of $G$ does not change if $C$ 
is replaced by $C \cap \mu$, where
\begin{equation} \label{e.mu}
\mu := \mu_{n_1} \times \dots \times \mu_{n_r}; 
\end{equation}
see Corollary~\ref{cor1.appendix}.  
Thus we will assume throughout that $C \subset \mu$. Unless otherwise
specified, we will also assume that $n_1 = p^{a_1}, \dots, n_r = p^{a_r}$
are powers of the same prime $p$. Here $a_1, \dots, a_r \geqslant 1$ are
integers.  Under these assumptions,
instead of $X(\bbG_m^r/C) \subset \bbZ^r$ we will 
consider the subgroup of 
$X(\mu) = (\bbZ/p^{a_1} \bbZ) \times \dots \times (\bbZ/p^{a_r} \bbZ)$
given by 
\begin{equation} \label{e.C}
\Code(C) := X(\mu/C) = \{ (m_1, \dots, m_r) \in X(\mu) \;  
| \; \tau_1^{m_1} \dots \tau_r^{m_r} = 1  \; \; 
\forall (\tau_1, \dots, \tau_r) \in C \} \, .  
\end{equation}
In other words, $\Code(C)$ consists of those characters
of $\mu$ which vanish on $C$. The symbol ``$\Code$" indicates that 
we will view this group as 
an error-correcting code.  In particular, we will 
define the Hamming weight $\pw(y)$ of 
\[ y = (m_1, \dots, m_r) \in 
(\bbZ/ p^{a_1} \bbZ) \times \dots \times (\bbZ/ p^{a_r} \bbZ) \]
as follows. Write $m_i :=  u_i p^{e_i}$ with $u_i \in (\Z/\p^{a_i}\Z)^*$ 
and $0 \leq e_i \leq a_i$. 
Then 
\[ \pw(y) := \sum_{i=1}^r(a_i-e_i) \, . \]
Our main results relate $\ed(G)$ to coding-theoretic invariants of
$\Code(C)$, such as its weight distribution; cf. 
also Corollary~\ref{cor2.appendix}.  For an introduction 
to error-correcting coding theory, see~\cite{macwilliams-sloane}. 

At this point we should warn the reader that our notions of 
error-correcting code and Hamming weight are somewhat unusual.
In coding-theoretic literature (linear) codes are usually defined as
linear subspaces of $\mathbb F_q^n$, where $\mathbb F_q$ is  
the field of $q$ elements. In this paper by a code we will mean 
an additive subgroup of
$(\bbZ/ p^{a_1} \bbZ) \times \dots \times (\bbZ/ p^{a_r} \bbZ)$.
Nevertheless, in an important special case, where 
$a_1 = \dots = a_r = 1$, our codes are linear codes of 
length $r$ over $\mathbb F_p$ in the usual sense of error-correcting 
coding theory, and our definition of the Hamming weight coincides 
with the usual definition.

\begin{thm}\label{thm.main1} 
Let $p$ be a prime,
$G := (\GL_{p^{a_1}} \times \dots \times \GL_{p^{a_r}})/C$, where
$C \subset  \mu_{p^{a_1}} \times \dots \times \mu_{p^{a_r}}$ is 
a central subgroup, and
$y_1, \dots, y_t$ be a minimal basis for $\Code(C)$; 
see Definition~{\rm \ref{def.minimal-basis}}. Then

\smallskip
(a) $\ed_p(G) \geqslant 
\left(\sum_{i=1}^tp^{\pw(y_i)}\right) - p^{2a_1} - \dots - p^{2 a_r} + r - t$,

\smallskip
(b) $\ed(G) \leqslant 
\left(\sum_{i=1}^tp^{\pw(y_i)}\right) - t +  \ed(\overline{G}) \;$
and 
$\; \ed_p(G) \leqslant 
\left(\sum_{i=1}^tp^{\pw(y_i)}\right) - t +  \ed_p(\bG)$,

\noindent
where $\overline{G} := \PGL_{p^{a_1}} \times \dots \times \PGL_{p^{a_r}}$.
\end{thm}

Although the upper and lower bounds 
of Theorem~\ref{thm.main1} never meet, for many 
central subgroups $C \subset \mu \subset G$, the term 
$\ds \sum_{i=1}^t p^{\pw(y_i)}$ is much larger than 
any of the other terms appearing in the above inequalities and
may be viewed as giving the asymptotic value of $\ed(G)$.
(In particular, note that in view of~\eqref{e.procesi},
\begin{equation} \label{e.product} \ed_p(\overline{G}) \leqslant 
\ed(\overline{G}) \leqslant \ed(\PGL_{p^{a_1}}) + \dots + \ed(\PGL_{p^{a_r}})
\leqslant p^{2a_1} + \dots + p^{2a_r}.) \end{equation} 
Under additional assumptions on $C$, we 
will determine $\ed(G)$ exactly; see Theorem~\ref{thm.main2}.

The fact that we can determine $\ed(G)$ for many choices of $C$,
either asymptotically or exactly, was rather surprising to us,
given the wide gap between the best known upper and lower 
bounds on $\ed(G)$ in the simplest case, where $r= 1$; 
see~\eqref{e.bm}.
Our informal explanation of this surprising phenomenon is 
as follows. If $\Code(C)$ can be generated by vectors 
$y_1, \dots, y_t$ of small weight, then 
$\ds \sum_{i=1}^tp^{\pw(y_i)}$ no longer dominates
the other terms.  In particular, this always happens 
if $r \leqslant 2$.  In such cases the value of $\ed(G)$ 
is controlled by the more subtle ``lower order effects", 
which are poorly understood. 
 
To state our next result, we will need the following terminology.
Suppose $2 \leqslant n_1 \leqslant \dots \leqslant n_t$ and
$z = (z_1, \dots, z_r) \in 
(\bbZ/ n_1 \bbZ) \times \dots \times (\bbZ/ n_r \bbZ)$, where
$z_{j_1}, \dots, z_{j_s} \neq 0$ for some 
$1 \leqslant j_1 < \dots < j_s \leqslant r$
and $z_j = 0$ for any $j \not \in \{j_1, \dots, j_r \}$.  
We will say that $z$ is {\em balanced} if 

\smallskip
(i) $n_{j_s} \leqslant \dfrac{1}{2}
n_{j_1} n_{j_2} \dots n_{j_{s-1}}$ and

\smallskip
(ii) $(n_{j_1}, \dots, n_{j_s}) \neq (2, 2, 2, 2), (3, 3, 3)$ 
or $(2, n, n)$ for any $n \geqslant 2$.

\smallskip
\noindent
Note that condition (i) can only hold if $s \geqslant 3$. In particular,
$(\bbZ/ n_1 \bbZ) \times \dots \times (\bbZ/ n_r \bbZ)$ 
has no balanced elements if $r \leqslant 2$.
In the sequel we will usually assume that $n_1, \dots, n_r$ are 
powers of the same prime $p$. In this situation condition
(ii) is vacuous, unless $p = 2$ or $3$.

\begin{thm} \label{thm.main2} Let $p$ be a prime,
$G := (\GL_{p^{a_1}} \times \dots \times \GL_{p^{a_r}})/C$, where
$a_r \geqslant \dots \geqslant a_1 \geqslant 1$ are integers,
and $C$ is a subgroup of $\mu$, as in~\eqref{e.mu}.  
Assume that the base field $k$ is of characteristic zero and
$\Code(C)$ has a minimal basis $y_i = (y_{i1}, \dots, y_{ir})$, 
$i = 1, \dots, t$ satisfying the following conditions: 

\smallskip
(a) $y_{ij} = -1$, $0$ or $1$ in $\Z/p^{a_j}\Z$, 
for every $i = 1, \dots, t$ and $j = 1, \dots, r$.

\smallskip
(b) For every $j = 1, \dots, r$ there exists an $i \in \{ 1, \dots, t \}$
such that $y_i$ is balanced and $y_{ij} \neq 0$.

\smallskip
\noindent
Then 
$\ds \ed(G) = \ed_p(G) = 
\left(\sum_{i=1}^t p^{\pw(y_i)}\right) - p^{2a_1} - \dots - p^{2a_r} + r - t$.
\end{thm}


Specializing Theorem~\ref{thm.main2}
to the case where $\Code(C)$ is generated by the single
element $(1, \dots, 1)$, we obtain the following.

\begin{thm} \label{thm.main3}
Let $a_r \geqslant a_{r-1} \geqslant \dots \geqslant a_1 \geqslant 1$ 
be integers and $\Func \colon \CatFk \to \CatS$ be 
the covariant functor given by 
\[ \ds \Func(K) := \left\{
\begin{array}{l} 
\text{isomorphism classes of $r$-tuples $(A_1, \dots, A_r)$ of 
central simple $K$-algebras } \\
\text{such that $\deg(A_i)=p^{a_i}\; \forall$ $i = 1, \dots, r$, 
and $A_1 \otimes \ldots \otimes A_r$ is split over $K$.}\\
\end{array}
\right\} \]

\smallskip
(a) If $\ds a_r \geqslant a_1 + \dots + a_{r-1}$, 
then $\ed(\Func) = \ed(\PGL_{p^{a_1}} \times
\dots \times \PGL_{p^{a_{r-1}}})$ and 
\[ \ed_p(\Func) = \ed_p(\PGL_{p^{a_1}} \times
\dots \times \PGL_{p^{a_{r-1}}}) \, . \]
In particular, $\ed(\Func) 
\leqslant p^{2a_1} + \dots + p^{2a_{r-1}}$.

\smallskip
(b) Assume that $\Char(k) = 0$, $\ds a_r < a_1 + \dots + a_{r-1}$, 
and $(p^{a_1}, \dots, p^{a_r})$ is not of the form 
$(2,2,2,2)$, $(3,3,3)$ or $(2, 2^a, 2^a)$, for any $a \geqslant 1$.
Then 
\begin{equation} \label{e.main3}
\ds \ed(\Func) = \ed_p(\Func) = 
p^{a_1 + \dots + a_r} - \sum_{i=1}^rp^{2a_i} +r-1 \, .
\end{equation}

(c) If $(p^{a_1}, \dots, p^{a_r}) = (2, 2, 2)$, then
$ \ds \ed(\Func) = \ed_2(\Func) = 3$. 
\end{thm}

Here part (c) treats the smallest of the exceptional cases 
in part (b). Note that in this case $p = 2$, $r = 3$ and $a_1 = a_2 = a_3 = 1$. 
Thus $p^{a_1 + \dots + a_r} - \sum_{i=1}^r p^{2a_i} +r-1 = -2$, and
formula~\eqref{e.main3} fails.  The values of $\ed(\Func)$ and 
$\ed_p(\Func)$ in the other exceptional cases, where 
$(p^{a_1}, \dots, p^{a_r}) = (2,2,2,2)$, $(3,3,3)$, or $(2, 2^a, 2^a)$ 
for some $a \geqslant 2$, remain open.

The results of this paper naturally lead to combinatorial questions,
which we believe to be of intependent interest but will not 
address here.  For each code (i.e. subgroup)
$X \subset (\bbZ/ p^{a_1} \bbZ) \times \dots \times (\bbZ/ p^{a_r} \bbZ)$ 
of rank $t$, let $(w_1, \dots, w_t)$ be the minimal profile 
of $X$ with respect to the Hamming weight function, in the sense 
of Proposition~\ref{prop.minimal-basis}. 
That is, $w_i = \pw(y_i)$, where $y_1, \dots, y_t$ is a minimal basis of $X$.  
Fixing $p$, $a_1 \leqslant \dots \leqslant a_r$ 
and $t$, and letting $X$ range over
all possible codes with these parameters,

\begin{itemize}
\item
\smallskip
What is the lexicographically largest profile $(w_1, \dots, w_t)$?

\item
\smallskip
What is the maximal value of $w_t$?

\item
\smallskip
What is the probability that $w_1 = \dots = w_t$?

\item
\smallskip
What is the maximal value of $p^{w_1} + \dots + p^{w_t}$? 

\item
\smallskip
What is the average value of $p^{w_1} + \dots + p^{w_t}$? 

\item
\smallskip
What is the probability that $w_t > 2a_r$? 
\end{itemize}

\smallskip
\noindent
Note that the expression $p^{w_1} + \dots + p^{w_t}$ appears 
in the formulas given in Theorem~\ref{thm.main1}. 
For large $p$ the condition that $w_t > 2a_r$ 
makes  $p^{w_1} + \dots + p^{w_t}$ the dominant term in
these formulas.
To the best of our knowledge, questions of this type 
(focusing on the minimal profile of a code, rather than 
the minimal weight) have not been 
previously investigated by coding theorists even in the case, where
$a_1 = \dots = a_r = 1$.

The rest of this paper is structured as follows.
In Section~\ref{sect.central} we prove general bounds on the essential 
dimension of certain central extensions of algebraic groups. 
These bounds will serve as 
the starting point for the proofs of the main theorems. 
To make these bounds explicit for groups of the form
$(\GL_{p^{a_1}} \times \dots \times \GL_{p^{a_r}})/C$
we introduce and study the notion of a minimal basis 
in Section~\ref{sect.minimal-basis}. 
Theorem~\ref{thm.main1}, \ref{thm.main2} and~\ref{thm.main3}  
are then proved in 
Sections~\ref{sect.proof-of-main1},~\ref{sect.proof-of-main2} 
and~\ref{sect.proof-of-main3}, respectively. The Appendix 
by Athena Nguyen contains an explicit description of the
Galois cohomology of groups of the form~\eqref{e.G}. This 
description and its corollaries are used throughout the paper.  

\section{Essential dimension and central extensions}
\label{sect.central}

Let $T = \bbG_m^r$ be a split $k$-torus of rank $r$, and
\begin{equation} \label{e.exact-sequence}
1 \to T \to G \to \overline{G} \to 1 
\end{equation}
be a central exact sequence of affine algebraic groups. This sequence 
gives rise to the exact sequence of pointed sets
\[ H^1(K, G) \to H^1(K, \overline{G}) \stackrel{\partial} \to H^2(K, T) \]
for any field extension $K$ of the base field $k$. Any character
$x \colon T \to \bbG_m$, induces a homomorphism $x_* \colon H^2(K, T) \to
H^2(K, \bbG_m)$.  We define $\ind^{x}(G, T)$ as the maximal
index of $x_* \circ \partial_K(E) \in H^2(K, T)$, where
the maximum is taken over all field extensions $K/k$ and over all
$E \in H^1(K, \overline{G})$. In fact, this maximal value is always 
attained in the case where $E = E_{\rm vers} \to \Spec(K)$ is a versal 
$G$-torsor (for a suitable field $K$). That is,
\begin{equation} \label{e.versal}
 \ind^{x}(G, T) = \ind (x_* \circ \partial_K(E_{\rm vers}))
\end{equation}
for every $x \in X(T)$; see~\cite[Theorem 6.1]{merkurjev-survey}.
Finally, we set
\begin{equation} \label{e.ind(G, T)}
\ind(G, T) : = \min \; \{ \sum_{i= 1}^r \, \ind^{x_i}(G, T) \, | \, 
\text{$x_1, \dots, x_r$ generate $X(T)$} \} \, . 
\end{equation}
Our starting point for the proof of the main
theorems is the following proposition. 

\begin{prop} \label{prop.central-extension} Assume that the image of
every $E \in H^1(K, \overline{G})$ under
\[ \partial \colon H^1(K, \overline{G}) \to H^2(K, T) \]
is $p$-torsion for every field extension $K/k$.  Then

\smallskip
(a) $\ed_p(G) \geqslant \ind(G, T) - \dim(G)$,

\smallskip
(b) $\ed(G) \leqslant \ind(G, T) + \ed(\overline{G}) - r \quad$
and 
$\quad \ed_p(G) \leqslant \ind(G, T) + \ed_p(\overline{G}) - r$.
\end{prop}

These bounds are variants of results that have previously appeared 
in the literature.
Part (a) is a generalization of~\cite[Corollary 4.2]{brv2} 
(where $r$ is taken to be $1$). In the case 
where $T$ is $\mu_p^r$, rather than $\bbG_m^r$, a variant of
part (a) is proved in~\cite[Theorem 4.1]{icm} (see 
also~\cite[Theorem 6.2]{merkurjev-survey}) and a variant of
part (b) in ~\cite[Corollaries 5.8 and 5.12]{merkurjev-survey}.

Our proof of Proposition~\ref{prop.index} proceeds along the same
lines as these earlier proofs; it relies on the notions 
of essential and canonical dimension of a gerbe (for 
which we refer the reader to~\cite{brv2} 
and~\cite{merkurjev-survey}), and the computation of 
the canonical dimension of a product of $p$-primary
Brauer-Severi varieties in~\cite[Theorem 2.1]{km-inventiones}.
In fact, the argument is easier for $T = \bbG_m^r$
than for $\mu_p^r$. In the former case (which 
is of interest to us here) the essential dimension
of a gerbe banded by $T$ is readily expressible in
terms of its canonical dimension (see formula~\eqref{e.ed=cd}
below), while an analogous formula
for gerbes banded by $\mu_p^r$ requires a far greater effort to prove.
(For $r = 1$, compare the proofs of
parts (a) and (b) of~\cite[Theorem 4.1]{brv2}. For arbitrary 
$r \geqslant 1$, 
see~\cite[Theorem 3.1]{km-inventiones} or
\cite[Theorem 5.11]{merkurjev-survey}.)  

\begin{proof} If $K/k$ is a field, and $E \in H^1(K, G)$, i.e.  
$E \to \Spec(K)$ is a $\overline{G}$-torsor, 
then the quotient stack
$[E/G]$ is a gerbe over $\Spec(K)$ banded by $T$. 
By~\cite[Corollary 3.3]{brv2} and~\cite[Corollary 5.7]{merkurjev-survey},
$\ed(G) \geqslant \max_{K, E} \, \ed([E/G]) - \dim(\overline{G})$ 
and similarly 
\[ \ed_p(G) \geqslant \max_{K, E} \, \ed_p([E/G]) - 
\dim(\overline{G}) \, , \]
where the maximum is taken over all field extensions $K/k$
and all $E \in H^1(K, \overline{G})$. On the other hand, 
by~\cite[Example 3.4(i)]{lotscher}
\[ \ed(G) \leqslant \ed(\overline{G}) + \max_{K, E} \, \ed([E/G])  
\quad \text{and} \quad
\ed_p(G) \leqslant \ed_p(\overline{G}) + \max_{K, E} \, \ed_p([E/G]) \, ; \]  
see also~\cite[Corollary 5.8]{merkurjev-survey}.
Since $\dim(G) = \dim(\overline{G}) + r$, it remains to show that 
\begin{equation} \label{e.gerbe-index}
\max_{K, E} \, \ed([E/G]) =
\max_{K, E} \, \ed_p([E/G]) = \ind(G, T) - r \, . 
\end{equation}

Choose a $\bbZ$-basis $x_1, \dots, x_r$ for the character group 
$X(T) \simeq \bbZ^r$ and let $P := P_1 \times \dots \times P_r$, where
$P_i$ is the Brauer--Severi variety associated to $(x_i)_* \circ \partial
(E) \in H^2(K, \bbG_m)$.  Since $T$ is a special group 
(i.e., every $T$-torsor over every field $K/k$ is split), 
the set $[E/G](K)$ of isomorphism classes of $K$-points 
of $[E/G]$ consists of exactly one element 
if $P(K) \neq \emptyset$ and is empty otherwise. Thus
\begin{equation} \label{e.ed=cd} 
\text{$\ed([E/G]) = \cd(P)$ and $\ed_p([E/G]) = \cd_p(P)$,} 
\end{equation}
where $\cd(P)$ denotes the canonical dimension of $P$.  
(The same argument is used in the proof of \cite[Theorem 4.1(a)]{brv2} in
the case, where $r = 1$.)  Since we are assuming that
$\partial(E)$ is $p$-torsion, the index of each Brauer-Severi variety
$P_i$ is a power of $p$. Thus by~\cite[Theorem 2.1]{km-inventiones},
\[ \cd(P) = \cd_p(P) =  \min \; \{ \sum_{i= 1}^r \, 
\ind ((x_{i})_* \circ \partial_K(E)) \, | \, 
\text{$x_1, \dots, x_r$ generate $X(T)$} \} - r  \, ; \] 
see also~\cite[Theorem 4.14]{merkurjev-survey}.
Taking $E := E_{\rm vers}$ to be a versal $\overline{G}$-torsor, 
we obtain
\[ \cd(P) = \cd_p(P) =  \min \; \{ \sum_{i= 1}^r \, 
\ind^{x_i}(G, T)) \, | \, 
\text{$x_1, \dots, x_r$ generate $X(T)$} \} - r \, ; \]
see~\eqref{e.versal}.  By the definition~\eqref{e.ind(G, T)} 
of $\ind(G, T)$, the last formula can be rewritten as
$\cd(P) = \cd_p(P) =  \ind(G, T) - r$.  Combining these equalities
with~\eqref{e.ed=cd}, we obtain~\eqref{e.gerbe-index}. 
\end{proof}

\begin{rem} \label{rem.central-extension}
Our strategy for proving Theorem~\ref{thm.main1}
will be to apply Proposition~\ref{prop.central-extension}
to the exact sequence~\eqref{e.exact-sequence} with
$G = (\GL_{p^{a_1}} \times \dots \times \GL_{p^{a_r}})/C$,
and $T := \bbG_m^r/C$.  The only remaining issue is to find 
an expression for $\ind(G, T)$ in terms of $\Code(C)$. 

Usually, the term $\ind(G, T)$ is computed using the formula
$\ind^{x}(G, T) = \gcd \, \dim(\rho)$, as $\rho \colon G \to \GL(V)$ 
ranges over all finite-dimensional representations of $G$, such that 
$\tau \in T$ acts on $V$ via scalar multiplication by $x(\tau)$.
See, for example,~\cite[Theorem 4.4]{km-inventiones} 
or~\cite[Theorem 6.1]{merkurjev-survey} or~\cite[Theorem 3.1]{lmmr-crelle}.
We will not use this approach in the present paper. 
Instead, we will compute the values 
of $\ind^{x}(G, T)$ and $\ind(G, T)$ directly 
from the definition, using the description of the connecting 
map $\partial \colon  H^1(K, \overline{G}) \to H^2(K, T)$ given 
by Theorem~\ref{thm.appendix}; see the proof 
of~Proposition~\ref{prop.index} below.
\end{rem}

\section{Minimal bases}
\label{sect.minimal-basis}

To carry out the program outlined in Remark~\ref{rem.central-extension}
we will need the notion of a {\em minimal basis}.
This section will be devoted to developing this notion.

The general setting is as follows.
Let $R$ be a local ring with maximal ideal $I \subset R$ and
$A$ be a finitely generated $R$-module.   
We will refer to a generating set $S \subset A$ as a {\em basis} if
no proper subset of $S$ generates $A$.  In the sequel we will specialize
$R$ to $\bbZ/ p^a \bbZ$ and  $A$ to a submodule of
$(\bbZ/p^{a_1} \bbZ) 
\times \dots \times (\bbZ/p^{a_t} \bbZ)$, where 
$a = \max(a_1, \dots, a_r)$. 
However, in this section it will be convenient for us to 
work over an arbitrary local ring $R$. 

Let $\pi \colon A \to A/IA$ be the natural projection.  
We will repeatedly appeal to
Nakayama's Lemma, which asserts that a subset $S \subset A$ generates
$A$ as an $R$-module if and only if $\pi(S)$ generates $A/IA$ as
an $R/I$-vector space; see~\cite[Section X.4]{lang}.

By a {\em weight function} on $A$ we shall mean any function
$w: A \ra \bbN$, where $\bbN$ denotes the set of non-negative integers.
We will fix $w$ throughout and  will sometimes 
refer to $\pw(y)$ as {\em the weight} of $y \in A$. 
For each basis $B=\{y_1, \dots, y_t\}$ of $A$, we will
define the {\em profile} of $B$ as
\[ \pw(B):=(\pw(y_1), \dots, \pw(y_t)) \in \bbN^t \, , \]
where $y_1, \dots, y_t$ are ordered so that $\pw(y_1) \leqslant
\pw(y_2) \leqslant \dots \leqslant \pw(y_t)$.
Let $\Prof(A) \subset \bbN^t$ denote the set of profiles of bases of $A$.

\begin{prop} \label{prop.minimal-basis}
$\Prof(A)$ has a unique minimal element with respect
to the partial order on $\bbN^t$ given by 
$(\alpha_1, \dots, \alpha_t) \preceq (\beta_1, \dots, \beta_t)$ 
if $\alpha_i \leqslant \beta_i$ for every $i = 1, \dots, t$.
\end{prop}

Note that since every descending chain in $(\Prof(A), \preceq)$ 
terminates, the unique minimal element is comparable to every 
element of $\Prof(A)$. 

\begin{proof}
We argue by contradiction. Set $t := \dim(A/IA)$. 
Suppose $X = \{ x_1, \dots, x_t \}$ 
and $Y = \{ y_1, \dots, y_t \}$ are bases of $A$ such that 
$\pw(X)$ and $\pw(Y)$ are distinct minimal elements of $\Prof(A)$. 
Let us order $X$ and $Y$ so that $\pw(x_1) \leqslant \dots \leqslant \pw(x_t)$ 
and $\pw(y_1) \leqslant \dots \leqslant \pw(y_t)$. Since $\pw(X) \neq \pw(Y)$, 
there exists an $s$ between $0$ and $t-1$ such that
\[ \text{$\pw(x_i)= \pw(y_i)$ for all $i = 1, \dots, s$} \]
but $\pw(x_{s+1}) \neq \pw(y_{s+1})$. After possibly interchanging $X$ and $Y$,
we may assume without loss of generality that $\pw(x_{s+1}) < \pw(y_{s+1})$. 

Let $\pi \colon A \to A/IA$ be the natural projection, as above.
By Nakayama's Lemma, $\pi(x_1), \dots, \pi(x_{s+1})$ 
are $R/I$-linearly independent in $A/IA$. Choose $t - s -1 $
elements of $Y$, say $y_{j_{s+2}}, \dots, y_{j_t}$ such that 
$\pi(x_1), \dots, \pi(x_{s+1}), \pi(y_{j_{s+2}}), \dots, \pi(y_{j_{t}})$ 
form an $R/I$-basis of $A/IA$.  After permuting $y_{j_{s+2}}, \dots, 
y_{j_t}$, we may assume that
$\pw(y_{j_{s+2}}) \leqslant \dots \leqslant \pw(y_{j_t})$. 
Applying Nakayama's lemma once again, we see that 
$Z = \{x_1, \dots, x_{s+1}, y_{j_{s+2}}, \dots, y_{j_{t}} \}$ is
a basis of $A$. 

We claim that $\pw(Z) \prec \pw(Y)$, where the inequality is strict.
Since we assumed that $\pw(Y)$ is minimal in $\Prof(A)$, this claim leads 
to a contradiction, thus completing the proof of 
Proposition~\ref{prop.minimal-basis}.
 
To prove the claim, let $z_1, \dots, z_t$, be the elements of $Z$,
in increasing order of their weight: 
$\pw(z_1) \leqslant \pw(z_2) \leqslant \dots \leqslant \pw(z_t)$. 
It suffices to show that $\pw(z_i) \leqslant \pw(y_i)$ for every $i = 1, \dots, t$,
and $\pw(z_{s+1}) < \pw(y_{s+1})$.
Let us consider three cases.

\smallskip
(i) $i \leqslant s$. Since
\[ \pw(x_1) = \pw(y_1) \leqslant \pw(x_2) = \pw(y_2) \leqslant  \dots 
\leqslant \pw(x_i) = \pw(y_i)  \, , \]
$Z$ has at least $i$ elements
whose weight is $\leqslant \pw(y_i)$, namely $x_1, \dots, x_i$.
Thus $\pw(z_i) \leqslant \pw(y_i)$. 

\smallskip
(ii) $i = s+1$. $Z$ has at least $s + 1$ elements, namely $x_1, \dots, x_{s+1}$
whose weight is at most $\pw(x_{s+1})$. Hence,
$\pw(z_{s+1}) \leqslant \pw(x_{s+1}) < \pw(y_{s+1})$, as desired.

\smallskip
(iii) $i > s+1$. Recall that both $y_1, \dots, y_t$ and
$y_{j_{s+2}}, \dots, y_{j_t}$ are arranged in weight-increasing 
order. For any $i \geqslant s+2$ there are at least $t-i + 1$ elements
of $Y$ whose weight is $\geqslant \pw(y_{j_i})$, namely  
$y_{j_i}, y_{j_{i+1}}, \dots, y_{j_t}$. Thus
\[ \pw(y_{j_i}) \leqslant \pw(y_i) \]
for any $i = s+2, \dots, t$.
Consequently, $Z$ has at least $i$ elements of weight $\leqslant
\pw(y_i)$, namely $x_1, \dots, x_{s+1}, y_{j_{s+2}}, \dots, y_{j_i}$. 
Hence,  $\pw(z_i) \leqslant \pw(y_i)$, as desired.

\smallskip
This completes the proof of the claim and hence
of Proposition~\ref{prop.minimal-basis}.
\end{proof}

\begin{defn} \label{def.minimal-basis} 
We will say that a basis $y_1, \dots, y_t$ of $A$ is 
{\em minimal} if its profile is the minimal element of $\Prof(A)$,
as in Proposition~\ref{prop.minimal-basis}. Note that a minimal basis 
in $A$ is usually not unique; however any two minimal bases have 
the same profile in $\bbN^t$. 
\end{defn}

\begin{rem} \label{rem:greedyalgorithm}
We can construct a minimal basis of $A$ using the following 
``greedy algorithm".
Select $y_1 \in A$ of minimal weight, subject to the condition
that $\pi(y_1) \neq 0$. Next select $y_2$ of minimal weight,
subject to the condition that $\pi(y_1)$ and $\pi(y_2)$ are 
$R/I$-linear independent in $A/IA$. Then select $y_3$ 
of minimal weight, subject to the condition that $\pi(y_1), \pi(y_2)$ and
$\pi(y_3)$ are $R/I$-linear independent in $A/IA$. Continue recursively.
After $t = \dim_{R/I}(A/IA)$ steps, 
we obtain a minimal basis $y_1, \dots, y_t$ for $A$.
\end{rem}

\begin{ex}
Set $R :=\F_p$, $I := (0)$, $G$ a finite $p$-group, $D:= Z(G)[p]$ the 
subgroup of $p$-torsion elements of the center $Z(G)$, 
and $A :=X(D)$ the group of characters of $D$. 
For $x \in A$, define $\pw(x)$ to be the minimal 
dimension of a representation
$G \to \GL(V_{x})$, such that $D$ 
acts on $V_{x}$ via scalar multiplication by $x$. 
If $\{x_1, \dots, x_t \}$ is a minimal basis of $A$, 
then $V_{x_1} \oplus \dots \oplus V_{x_t}$ is a faithful 
representation of $G$ of minimal dimension; see 
\cite[Remark 4.7]{km-inventiones}. 
\end{ex}

\section{Conclusion of the proof of Theorem~\ref{thm.main1}}
\label{sect.proof-of-main1}

Recall that we are interested in the essential dimension of
the group 
\[ G = (\GL_{p^{a_1}} \times \dots \times \GL_{p^{a_r}})/C \, , \]
where $C$ is a subgroup of 
$\mu := \mu_{p^{a_1}} \times \dots \times \mu_{p^{a_r}}$.
We will think of the group of characters $X(\bbG_m^r)$ as $\bbZ^r$ by
identifying the character $x(\tau_1, \dots, \tau_r) = \tau_1^{m_1} \dots \tau_r^{m_r}$ 
with $(m_1, \dots, m_r) \in \bbZ^r$. Characters
of $T:= G_m^r/C$ are identified in this manner with the $r$-tuples 
$(m_1, \dots, m_r) \in \bbZ^r$ such that $\tau_1^{m_1} \dots \tau_r^{m_r} = 1$
for every $(\tau_1, \dots, \tau_r) \in C$. The relationship among 
these character groups is illustrated by the following diagram 
 \[ \xymatrix{ 
& X(\bbG_m^r/C) \ar@{->}[d]
 \ar@{^{(}->}[r] & X(\bbG_m^r) \ar@{->}[d]^{\pi} \ar@{=}[r]  
                                        & \quad \bbZ \times \dots \times \bbZ 
\text{ ($r$ times)} \ar@{->}[d]^{\pi}  \\
\Code(C) \ar@{=}[r] & X(\mu/C) \ar@{^{(}->}[r] & X(\mu) \ar@{=}[r] 
    & (\bbZ/p^{a_1} \bbZ) \times \dots \times (\bbZ/p^{a_r} \bbZ).} \]
Here $\Code(C)$ is as in~\eqref{e.C} and $\pi$ is the natural projection,
given by restricting a character from $\bbG_m^r$ to $\mu$. 

Our proof of Theorem~\ref{thm.main1} will be based on
the strategy outlined in Remark~\ref{rem.central-extension}.
In view of Proposition~\ref{prop.central-extension} 
it suffices to establish the following:

\begin{prop} \label{prop.index}
Consider the central exact sequence 
\begin{equation} \label{e.central}
1 \to T \to G \to \overline{G} \to 1 \, , 
\end{equation}
where $G = (\GL_{p^{a_1}} \times \dots \times \GL_{p^{a_r}})/C$,
$C$ is a subgroup of $\mu := \mu_{p^{a_1}} \times \dots \times \mu_{p^{a_r}}$,
$T := \bbG_m^r/C$ and 
$\overline{G} := \PGL_{p^{a_1}} \times \dots \times \PGL_{p^{a_r}}$.

\smallskip
(a) If $x \in X(T)$ and $y = \pi(x) \in \Code(C)$ then
$\ind^{x}(G, T) = p^{\pw(y)}$.

\smallskip
(b) $\ind(G, T) = p^{\pw(z_1)} + \dots + p^{\pw(z_t)} + r - t$, where
$z_1, \dots, z_t$ is a minimal basis of $\Code(C)$.
\end{prop}

\begin{proof}[Proof of Proposition~\ref{prop.index}(a)]
Consider the connecting map $\partial \colon H^1(K, \overline{G}) \to
H^2(K, T)$ associated to the central exact sequence~\eqref{e.central}.
Given a character $x \colon T \to \bbG_m$,  
$x(\tau_1, \dots, \tau_r) = \tau_1^{m_1} \dots \tau_r^{m_r}$,
$\ind^x(G, T)$ is, by definition, the maximal 
value of $\ind(x_* \partial(E))$, as $K$ ranges over all fields
containing $k$ and $E$ ranges over $H^1(K, \overline{G})$. In this case,
$\overline{G} = \PGL_{p^{a_1}} \times \dots \times \PGL_{p^{a_r}}$,
and thus $H^1(K, \overline{G})$ is the set of $r$-tuples $(A_1, \dots, A_r)$
of central simple algebras, where the degree of $A_i$ is $p^{a_i}$. 
The group $H^2(K, \bbG_m)$ is naturally identified with
the Brauer group $\Br(K)$, and the map $x_* \partial$ takes an $r$-tuple
$(A_1, \dots, A_r)$, as above, to the Brauer class of
$A := A_1^{m_1} \otimes \dots \otimes A_r^{m_r}$.

Since $\deg(A_i) = p^{a_i}$, the Brauer class of $A$ depends only on 
\[ y = \pi(x) = (m_1 \mod{p^{a_1}} \, , \dots , \, m_r \mod{p^{a_r}}) \in 
    (\bbZ/p^{a_1} \bbZ) \times \dots \times (\bbZ/p^{a_r} \bbZ)) \, . \]
Moreover, if $m_i \equiv u_i p^{e_i} \pmod{p^{a_i}}$,
where $u_i$ is prime to $p$ 
and $0 \leqslant e_i \leqslant a_i$ then $\ind(A_i^{\otimes m_i}) \leqslant 
p^{a_i - e_i}$.  
Now recall that $\pw(y)$ is defined as $\sum_{i= 1}^r (a_i - e_i)$. Thus
\[ \ind(A) \leqslant \prod_{i= 1}^r \ind(A_i^{\otimes{m_i}}) \leqslant
\prod_{i= 1}^r p^{a_i - e_i} = p^{\pw(y)} \, .  \]  
To prove the opposite inequality, we set $A_i$ to be the symbol 
algebra $(\alpha_i, \beta_i)_{p^{a_i}}$,
over the field $K = k(\zeta)(\alpha_1, \dots, \alpha_r, \beta_1, \dots, 
\beta_r)$, where
$\zeta$ is a primitive root of unity of degree $p^{\max(a_1, \dots, a_r)}$ 
and
$\alpha_1, \dots, \alpha_r, \beta_1, \dots, \beta_r$ 
are $2r$ independent variables over $k$. Writing $m_i = u_i p^{e_i}$, as above,
we see that $A_i^{m_i}$ is Brauer equivalent to 
$B_i:= (\alpha_i, \beta_i^{u_i})_{p^{a_i - e_i}}$ over $K$. An easy 
valuation-theoretic argument shows that $B := B_1 \otimes_K \dots \otimes_K B_t$
is a division algebra.  (In particular, the norm form 
of $B$ is a Pfister polynomial and hence, is anisotropic; see
\cite[Theorem~3.2 and Proposition~3.4]{hermite-joubert}.) Thus
\[ \ind(A) = \ind(B) = \ind(B_1) \cdot \ldots \cdot \ind(B_t) =  
p^{(a_1 -e_1) + \dots + (a_t - e_t)} = p^{\pw(y)} \, , \]
as desired.  We conclude that $\ind^x(G, T) \geqslant \ind(A) = p^{\pw(y)}$,
thus completing the proof of Proposition~\ref{prop.index}(a).
\end{proof}
 
Our proof of Proposition~\ref{prop.index}(b) will rely 
on the following elementary lemma.

\begin{lem} \label{lem.elementary} Let $p$ be a prime,
$M$ be a finite abelian $p$-group, and
$f \colon \bbZ^n \to M$ be a surjective $\bbZ$-module 
homomorphism for some
$n \geqslant 1$.  Then for every basis $y_1, \dots, y_t$ of $M$,
there exists a $\bbZ$-basis $x_1, \dots, x_n$ of $\bbZ^n$ 
and an integer $c$ prime to $p$, such that
$f(x_1) = c y_1, f(x_2) = y_2, \dots, f(x_t) = y_t$ 
and $f(x_{t+1}) = \dots = f(x_n) = 0$.
\end{lem}

\begin{proof} 
 By \cite[Theorem III.7.8]{lang} there exists a basis 
 $e_1, \dots, e_n$ of $\bbZ^n$ such that $\Ker(f)$ is generated by 
 $p^{d_i} e_i$ for some integers $d_1, \dots, d_t \geqslant 0$.  
 Since $M$ has rank $t$, we may assume without loss of generality 
 that $d_1, \dots, d_t \geqslant 1$ 
 and $d_{t+1} = \dots = d_n = 0$. That is, we may identify
 $M$ with $(\bbZ/p^{d_1} \bbZ) \times \dots \times (\bbZ/p^{d_t} \bbZ)$ 
 and assume that 
 \[ f (r_1, \dots, r_n) = (\, r_1 \! \! \! \mod{p^{d_1}}, \,  \dots \, , 
 \, r_t \! \! \! \mod{p^{d_t}})\quad \forall (r_1, \dots, r_n) 
 \in \bbZ^n \, .\] 
 It now suffices to lift $cy_1, \dots, y_t \in M$
 to a basis $x_1, \dots, x_t$ of $\bbZ^t$, for a suitable integer
$c$, prime to $p$. Indeed, if we 
manage to do this, then we will obtain a basis of $\bbZ^n$ of 
the desired form by appending 
\[ x_{t+1} := e_{t+1}, \dots, x_n := e_n \in \Ker(f) \]
to
$x_1, \dots, x_t$.  Thus we may assume that $n = t$. 

Now observe that $f \colon \bbZ^n \to M$, factors as
$\bbZ^n \to (\bbZ/p^d \bbZ)^n \to M$,
where $d:= \max(d_1, \dots, d_t)$.  
Lift each $y_i \in M$ to some $y_i' \in (\bbZ/p^d \bbZ)^n$.
By Nakayama's Lemma $y_1', \dots, y_n'$ form a   
$\bbZ/ p^d \bbZ$-basis of $(\bbZ/p^d \bbZ)^n$.
It now suffices to lift $cy_1', y_2', \dots, y_n'$ 
to a basis of $\bbZ^n$, for a suitable integer $c$, prime to $p$.
In other words, we may assume without loss of generality that 
$M = (\bbZ/p^d \bbZ)^n$, and 
$f \colon \bbZ^n \to (\bbZ/p^d \bbZ)^n$ 
is the natural projection.

Now suppose  $y_i = (y_{i1}, \dots, y_{in})$ for
some $y_{ij} \in \bbZ/p^d \bbZ$.  Since $y_1, \dots, y_m$
form a basis of $(\bbZ/p^d \bbZ)^n$, the matrix $A = (y_{ij})$ 
is invertible, i.e., $A \in \GL_n(\bbZ/p^d \bbZ)$.
After rescaling $y_1$ by $c := \det(A)^{-1}$ in $\bbZ/p^d \bbZ$, 
we may assume that
$\det(A) = 1$.  The lemma now follows from the surjectivity of 
the natural projection $\SL_t(\bbZ) \to \SL_t(\bbZ/p^d \bbZ)$; 
see~\cite[Lemma 1.38]{shimura}.
\end{proof}

\begin{proof}[Proof of Proposition~\ref{prop.index}(b)] 
By definition, $\ind(G, T)$ is the minimal value 
of $\ind^{x_1}(G, T) + \dots + \ind^{x_r}(G, T)$, as $x_1, \dots, x_r$
range over the bases of $X(T) \subset \bbZ^r$. By part (a), we can
rewrite this as
\[ \ind(G, T) = \min \{ p^{\pw(\pi(x_1))} + \dots + p^{\pw(\pi(x_r))}  \, | \, 
\text{$x_1, \dots, x_r$ is a $\bbZ$-basis of $X(T) \}$.} 
\]  
Here, as before, $\pi(x_i) \in \Code(C)$ is the restriction of $x_i$ 
from $T = \bbG_m^r/C$ to $\mu/C$. 

Let $z_1, \dots, z_t \in \Code(C)$ be a minimal basis, as in the statement
of the proposition.  We will prove part (b) by showing that 

\smallskip
(i) $p^{\pw(\pi(x_1))} + \dots + p^{\pw(\pi(x_r))}  \geqslant
p^{\pw(z_1)} + \dots + p^{\pw(z_t)} + r - t$ for every $\bbZ$-basis 
$x_1, \dots, x_r$ of $X(T)$, and

\smallskip
(ii) there exists a particular $\bbZ$-basis  
$x_1, \dots, x_r$ of $X(T)$ such that
$p^{\pw(\pi(x_1))} + \dots + p^{\pw(\pi(x_r))} = 
p^{\pw(z_1)} + \dots + p^{\pw(z_t)} + r - t$.

\smallskip
To prove (i), note that if $x_1, \dots, x_r$ form a $\bbZ$-basis of $X(T)$,
then $\pi(x_1), \dots, \pi(x_r)$ form a generating set for $\Code(C)$. 
By Nakayama's Lemma every generating set for $\Code(C)$ contains a basis.
After renumbering $x_1, \dots, x_r$ we may assume that
$\pi(x_1), \dots, \pi(x_t)$ is a basis of $\Code(C)$ and
$\pw(\pi(x_1)) \leqslant \dots \leqslant \pw(\pi(x_t))$.
By Proposition~\ref{prop.minimal-basis}, $\pw(z_i) \leqslant \pw(\pi(x_i))$ 
for every $i = 1, \dots, t$. Thus 
\[ p^{\pw(\pi(x_1))} + \dots + p^{\pw(\pi(x_r))} 
\geqslant p^{\pw(\pi(x_1))} + \dots + p^{\pw(\pi(x_t))} + 
\underbrace{p^0 + \dots + p^0}_{\text{\tiny $r-t$ times}} 
\geqslant p^{\pw(z_1)} + \dots + p^{\pw(z_t)} +  r - t \, . \]

To prove (ii), recall that by Lemma~\ref{lem.elementary} there exists
an integer $c$, prime to $p$, and a $\bbZ$-basis $x_1, \dots, x_r$ of
$X(T)$ such that $\pi(x_1) = cz_1, \pi(x_2) = z_2, \dots, \pi(x_t) = z_t$,
and $\pi(x_{t+1}) = \dots = \pi(x_r) = 0$. Since $c$ is prime to $p$,
$\pw(cz_1) = \pw(z_1)$. Thus for this particular choice of $x_1, \dots, x_r$, 
we have
\[ p^{\pw(\pi(x_1))} + \dots + p^{\pw(\pi(x_r))} =
p^{\pw(cz_1)} + p^{\pw(z_2)} \dots + p^{\pw(z_t)} + 
\underbrace{p^0 + \dots + p^0}_{\text{\tiny $r-t$ times}} =
p^{\pw(z_1)} + \dots + p^{\pw(z_t)} +  r - t \, . \]
as desired.
\end{proof}

\section{Proof of Theorem~\ref{thm.main2}}
\label{sect.proof-of-main2}

Consider the action of a linear algebraic group $\Gamma$ on an 
absolutely irreducible
algebraic variety $X$ defined over $k$. We say that 
a subgroup $S \subset \Gamma$ is a {\em stabilizer in general position} 
for this action if there exists a dense open subset $U \subset X$
such that the scheme-theoretic stabilizer $\Stab_{\Gamma}(x)$ is 
conjugate to $S$ over $\overline{k}$ for every $x \in U(\overline{k})$.
Here, as usual, $\overline{k}$ denotes the algebraic closure of $k$.   
In the sequel we will not specify $U$ and will simply say that 
$\Stab_{\Gamma}(x)$ is conjugate to $S$ for $x \in X(\overline{k})$ 
in general position.
Note that a stabilizer in general position $S$ for a $\Gamma$-action on $X$
does not always exist, and when it does, it is usually not unique.
However, over $\overline{k}$, $S$ is unique up to conjugacy.

For the rest of this section we will always assume that $\Char(k) = 0$.
A theorem of R.~W.~Richardson~\cite{richardson} tells us that under this 
assumption every linear action of a reductive group $\Gamma$ on 
a vector space $V$ has a stabilizer $S \subset \Gamma$ in general position.  
Note that in Richardson's paper~\cite{richardson}, $k$ is assumed 
to be algebraically closed. 
Thus a priori the subgroup $S$ and the open subset $U \subset V$, where all 
stabilizers are conjugate to $S$, are only defined over $\overline{k}$.
However, note that $U$ has only finitely many Galois translates.
After replacing $U$ by the intersection of all of these translates,
we may assume that $U$ is defined over $k$.  Moreover, we may take
$S := \Stab_G(x)$ for some $k$-point $x \in U(k)$ and thus assume that
$S$ is defined over $k$. 
For a detailed discussion of stabilizers in general position 
over an algebraically closed field of
characteristic zero, see~\cite[Section 7]{popov-vinberg}.

We will say that a $\Gamma$-action on $X$ is {\em generically free} 
if the trivial subgroup $S = \{ 1_{\Gamma} \} \subset \Gamma$ 
is the stabilizer in general position for this action. 

\begin{lem} \label{lem.central}
Let $\Gamma$ be a reductive linear algebraic group and
$\rho \colon \Gamma \to \GL(V)$ be a finite-dimensional representation.
If $\Stab_{\Gamma}(v)$ is central in $\Gamma$ 
for $v \in V$ in general position, 
then the induced action of $\Gamma/\Ker(\rho)$ on $V$ is generically free.
\end{lem}

\begin{proof}
Let $S \subset \Gamma$ be a stabilizer in general 
position for the $\Gamma$-action on $V$. Clearly $\Ker(\rho) \subset S$.
We claim that, in fact, $\Ker(\rho) = S$; the lemma easily follows from 
this claim. 

To prove the opposite inclusion, $S \subset \Ker(\rho)$, note that
under the assumption of the lemma, $S$ is central in $\Gamma$. 
Let $U \subset V$ be a dense open subset such that the stabilizer 
of every $v \in U(\overline{k})$ is conjugate to $S$. 
Since $S$ is central, $\Stab_{\Gamma}(v)$ is, in fact, 
equal to $S$. In other words, $S$ stabilizes every point in $U$ and 
thus every point in $V$. That is, $ S \subset \Ker(\rho)$, as claimed.
\end{proof}

Our interest in generically free actions in this section
has to do with the following fact: if there exists a 
generically free linear representation $G \to \GL(V)$ then
\begin{equation} \label{e.representation}
\ed(G) \leqslant \dim(V) - \dim(G) \, ; 
\end{equation}
see, e.g., \cite[(2.3)]{icm} or \cite[Proposition 3.13]{merkurjev-survey}.
This inequality will play a key role in our proof of Theorem~\ref{thm.main2}.

Now set $\Gamma := \GL_{n_1} \times \dots \times \GL_{n_r}$ and 
$\Gamma' :=\SL_{n_1} \times \dots \times \SL_{n_r}$. Let $V_{i}$ 
be the natural $n_i$-dimensional representation, 
$V_{i}^{-1}$ be the dual representation, and 
$V_{i}^0$ be the trivial $1$-dimensional representation of $\GL_{n_i}$.
For $\epsilon = (\epsilon_1, \dots, \epsilon_r)$, where 
each $\epsilon_i$ is $-1$, $0$ or $1$, we define $\rho_{\epsilon}$
to be the natural representation of $\Gamma$ on the
tensor product 
\begin{equation} \label{e.V_epsilon}
V_{\epsilon} = V_1^{\epsilon_1} 
\otimes \dots \otimes V_r^{\epsilon_r} \, .
\end{equation}

\begin{lem}\label{lem.ampopov} Suppose 
$2 \leqslant n_1 \leqslant \ldots \leqslant n_r \leqslant
\frac{1}{2} n_1 \dots n_{r-1}$, and
\[ \text{$(n_1, \dots, n_r) \neq (2, 2, 2, 2)$, $(3, 3, 3)$, or
$(2, n, n)$, for any $n \geqslant 2$.}
\] 
If $\epsilon = (\epsilon_1, \dots, \epsilon_r) \in \{ \pm 1 \}^r$,
then the induced action of $\Gamma /\Ker(\rho_{\epsilon})$ on
$V_{\epsilon}$ is generically free. 
\end{lem}

\begin{proof} 
By Lemma~\ref{lem.central} it suffices to  prove the following claim:
the stabilizer $\Stab_{\Gamma}(v)$ is central in $\Gamma$ for 
$v \in V_{\epsilon}$ in general position. To prove this claim, 
we may assume without loss of generality 
that $k$ is algebraically closed.

We first reduce to the case where $\epsilon = (1,\dots, 1)$. Suppose 
the claim is true in this case, and let $(\epsilon_1, \dots, \epsilon_r)
\in \{\pm1\}^r$. By choosing bases of $V_1, \dots, V_r$ we can identify 
$V_i$ with $V_i^{\epsilon_i}$ 
(we can take the identity map if $\epsilon_i=1$). Define an automorphism:
\begin{eqnarray*}
\sigma: \Gamma  & \to & \Gamma \\
(g_1, \dots, g_r) & \mapsto & (g_1^*, \dots, g_r^*)\\
\end{eqnarray*}
where 
\[ g_i^{*}= 
\begin{cases} g_i & \textrm{if $\epsilon_i=1$};\\ 
(g_i^{-1})^{T} &\textrm{if $\epsilon_i= -1$.} \end{cases} \]
Now $\rho_{(\epsilon_1, \dots, \epsilon_r)}$ 
is isomorphic to the representation 
$\rho_{(1, \dots, 1)} \circ \sigma$. 
Since the center of $\Gamma$ is invariant under $\sigma$,
we see that the claim holds for $\rho_{\epsilon}$ as well.

 From now on we will assume $\epsilon = (1, \dots, 1)$.
By \cite[Theorem 2]{P87}, 
\[ \Gamma / Z(\Gamma) = 
\PGL_{n_1} \times \dots \times \PGL_{n_r} = \Gamma'/Z(\Gamma') \]
acts generically freely on the projective space 
$\bbP(V_{\epsilon}) = V_{\epsilon}/Z(\Gamma)$.  
In other words for $v \in V_{\epsilon}$ in general position
the stabilizer in $\Gamma$ of the associated projective point 
$[v] \in \bbP(V_{\epsilon})$ is trivial. Hence,
the stabilizer of $v$ is contained in $Z(\Gamma)$;
cf. the exact sequence in~\cite[Lemma 3.1]{rv}. This completes the proof of
the claim and thus of Lemma~\ref{lem.ampopov}.  
\end{proof}

We are now ready to proceed with the proof of Theorem~\ref{thm.main2}.  
We begin by specializing $n_i$ to $p^{a_i}$ for every $i = 1, \dots, r$,
so that $\Gamma$ becomes $\GL_{p^{a_1}} \times \dots \times \GL_{p^{a_r}}$. 
Let \[ y_1, \dots, y_t \in (\bbZ/ p^{a_1} \bbZ) \times \dots \times 
(\bbZ/p^{a_r} \bbZ) \] be a basis of $\Code(C)$ satisfying the 
conditions 
of Theorem~\ref{thm.main2}.  Lift each $y_i = (y_{i1}, \dots, y_{ir})$ to 
$x_i := (x_{i1}, \dots, x_{ir}) \in \bbZ^r$ by setting 
$x_{ij} := -1$, $0$ or $1$, depending on whether 
$y_{ij}$ is $-1$, $0$ or $1$ in $\bbZ/p^{a_j} \bbZ$. 
(If $p^{a_j} = 2$, then we define each $x_{ij}$ to be $0$ or $1$.)
By Nakayama's Lemma the images of $y_1, \dots, y_t$ 
are $\bbF_p$-linearly independent 
in $\Code(C)/p\Code(C)$. Thus the integer vectors 
$x_1, \dots, x_t$ are $\bbZ$-linearly 
independent.   (Note that, unlike in the situation of 
Lemma~\ref{lem.elementary}, here 
it will not matter to us whether $x_1, \dots, x_t$ can be completed to
a $\bbZ$-basis of $\bbZ^r$.) 
We view each $x_i$ as a character $\bbG_m^r \to \bbG_m$ and
set 
\[ \widetilde{C} := \Ker(x_1) \cap \dots \cap \Ker(x_t) \subset \bbG_m^r \, . \]
Since $x_1, \dots, x_t$ are linearly independent, 
\begin{equation} \label{e.tildeC}
\dim(\widetilde{C}) = r - t.
\end{equation}
Set $G := \Gamma/C$ and  $\widetilde{G} := \Gamma/\widetilde{C}$.
By our construction, $\widetilde{C} \cap \mu = C$. 
Corollary~\ref{cor1.appendix} now tells us that
$\ed_p(G) \leqslant \ed(G) = \ed(\widetilde{G})$.
By Theorem~\ref{thm.main1}(a)
\[ \ed(G) \geqslant \ed_p(G) \geqslant
\left(\sum_{i=1}^t p^{\pw(y_i)}\right) - p^{2a_1} - \dots - p^{2a_r} + r - t . \]
It thus suffices to show that
$\ed(\widetilde{G}) \leqslant
\left(\sum_{i=1}^t p^{\pw(y_i)}\right) - p^{2a_1} - \dots - p^{2a_r} + r - t$
or equivalently,
\[ \ed(\widetilde{G}) \leqslant \left(\sum_{i=1}^t p^{\pw(y_i)}\right) - 
\dim(\widetilde{G}) \, ; \]
see~\eqref{e.tildeC}.
By~\eqref{e.representation}, in order to prove the last inequality
it is enough to construct a generically free linear representation of
$\widetilde{G}$ of dimension $\sum_{i=1}^t p^{\pw(y_i)}$. Such 
a representation is furnished by the lemma below.  

Recall that $x_i = (x_{i1}, \dots, x_{ir}) \in \bbZ^r$, where 
each $x_{ij} = -1$, $0$ or $1$, and $\rho_{x_i}$ 
is the natural representation of
$\Gamma := \GL_{p^{a_1}} \times \dots \times \GL_{p^{a_r}}$ 
on $V_{x_i} := V_1^{x_{i1}} \otimes \dots \otimes 
V_{r}^{x_{ir}}$,
as in~\eqref{e.V_epsilon}, with $\dim(V_i) = n_i = p^{a_i}$. 

\begin{lem} \label{lem.tensors}
Let $V = V_{x_1} \oplus \dots \oplus V_{x_t}$ and
$\rho := \rho_{x_1} \oplus \dots \oplus \rho_{x_t} \colon 
\Gamma \to \GL(V)$.
Then

\smallskip
(a) $\dim(V) =  p^{\pw(y_1)} + \dots + p^{\pw(y_t)}$,

\smallskip
(b) $\Ker(\rho) = \widetilde{C}$, and 

\smallskip
(c) the induced action of 
$\widetilde{G} = \Gamma /\widetilde{C}$ 
on $V$ is generically free.
\end{lem}

\begin{proof} For each $i = 1, \dots t$, we have
\[ \dim(V_{x_i}) = \prod_{x_{ij} \neq 0} \, p^{a_j} 
= \prod_{y_{ij} \neq 0} \, p^{a_j} = p^{\sum_{y_{ij} \neq 0} a_j} \, . \]
Since each $y_{ij} = -1$, $0$ or $1$, $\sum_{y_{ij} \neq 0} a_j = \pw(y_i)$.
Thus $\dim(V_{x_i}) = p^{\pw(y_i)}$, and part (a) follows.

Now choose $v_i \in V_{x_i}$ in general 
position and set $v := (v_1, \dots, v_r)$. 
We claim that $\Stab_{\Gamma} (v)$ is central in $\Gamma$.  

Suppose for a moment that this claim is established.
Since the center $Z(\Gamma) = \bbG_m^r$  
acts on $V_{x_i}$ via scalar multiplication by the character 
$x_i \colon \bbG_m^r \to \bbG_m$, we see that
\[ \Ker(\rho) = \Ker(\rho_{|\bbG^m}) = 
\Ker(x_1) \cap \dots \cap \Ker(x_t) = \widetilde{C} \, , \]
and part (b) follows. Moreover, by Lemma~\ref{lem.central},
the induced action of $ \Gamma/\Ker(\rho)$ 
on $V$ is generically free. By part (b), $\Ker(\rho) = \widetilde{C}$ 
and part (c) follows as well.

It remains to prove the claim.  Choose $v_i \in V_{x_i}$ in general 
position and assume that $g = (g_1, \dots, g_r)$ stabilizes
$v := (v_1, \dots, v_t)$ in $V$ for some $g_j \in \GL_{p^{a_j}}$.
Our goal is to show that $g_j$ is, in fact, central in $\GL_{p^{a_j}}$
for each $j = 1, \dots, r$. 

Let us fix $j$ and focus on proving that $g_j$ is central for this
particular $j$.
By assumption (b) of Theorem~\ref{thm.main2}, there exists
an $i = 1, \dots, t$ such that $y_i$ is balanced and $y_{ij} \neq 0$.
Let us assume that $y_{i j_1}, \dots, y_{i j_s} = \pm 1$ and
$y_{ih} = 0$ for every $h \not \in \{ j_1, \dots, j_r \}$
and consequently, $x_{i j_1}, \dots, x_{i j_s} = \pm 1$ and
$x_{ih} = 0$ for every $h \not \in \{ j_1, \dots, j_r \}$.
By our assumption, $j \in \{ j_1, \dots, j_s \}$.

The representation $\rho_{x_i}$ of
$\Gamma = \GL_{p^{a_1}} \times \dots \times \GL_{p^{a_r}}$ 
on 
\[ V_{x_i} :=  V^{x_{i1}} \otimes \dots \otimes V^{x_{it}} =
V^{x_{ij_1}} \otimes \dots \otimes V^{x_{ij_s}}  \]
factors through the projection 
$\Gamma \to \GL_{p^{a_{j_1}}} \times \dots \times \GL_{p^{a_{j_s}}}$.
Thus if $g = (g_1, \dots, g_r)$ stabilizes 
$v = (v_1, \dots, v_t) \in V$ then, in particular, $g$ stabilizes $v_i$ 
and thus $(g_{j_1}, \dots, g_{j_s})$ stabilizes $v_i$.

Since $y_i$ is assumed to be balanced, the conditions 
of Lemma~\ref{lem.ampopov} 
for the action of $\GL_{n_{j_1}} \times \dots \times \GL_{n_{j_s}}$ 
on $V_{x_i} = V^{x_{j_1}} \otimes \dots \otimes V^{x_{j_s}}$ are satisfied.
(Recall that here $n_i = p^{a_i}$.)  
Since $(g_{j_1}, \dots, g_{j_s})$ stabilizes $v_i \in V_{x_i}$ in
general position,
Lemma~\ref{lem.ampopov} tells us that $g_{j_1}, \dots, g_{j_s}$ are
central in $\GL_{n_{j_1}}, \dots \GL_{n_{j_s}}$, respectively.
In particular, $g_j$ is central in  $\GL_{n_j}$, as desired.
This completes the proof of Lemma~\ref{lem.tensors} and thus of
Theorem~\ref{thm.main2}.
\end{proof}

\section{Proof of Theorem~\ref{thm.main3}}
\label{sect.proof-of-main3}

Consider the central subgroups $\widetilde{C}$
and  $C$ of $\Gamma = \GL_{p^{a_1}} \times \dots \times \GL_{p^{a_r}}$
given by 
\[ \text{$\widetilde{C} = \{ (\tau_1, \dots, \tau_r) \in \bbG_m^r \; | \; 
\tau_1 \dots \tau_r = 1 \}$ and
$C = \{ (\tau_1, \dots, \tau_r) \in \mu \; | \; \tau_1 
\dots \tau_r = 1 \}$.} \]
Set $G : = \Gamma/C$ and $\widetilde{G} : = \Gamma/\widetilde{C}$.
Note that $C = \widetilde{C} \cap \mu$. Thus Theorem~\ref{thm.appendix}
and Corollary~\ref{cor1.appendix} 
tell us that the functors $H^1(-, G)$ and $H^1(-, \tilde{G})$ 
are both isomorphic to
\[ \ds \Func \colon K \mapsto \left\{
\begin{array}{l} 
\text{isomorphism classes of $r$-tuples $(A_1, \dots, A_r)$ of 
central simple $K$-algebras } \\
\text{such that $\deg(A_i)=p^{a_i}\; \forall i$, 
and $A_1 \otimes \ldots \otimes A_r$ is split over $K$.}\\
\end{array}
\right\} \]
In particular, $\ed(\widetilde{G}) = \ed(G) = \ed(\Func)$ and 
$\ed_p(\widetilde{G}) = \ed_p(G) = \ed_p(\Func)$. 
We are now ready to proceed with the proof of Theorem~\ref{thm.main3}.

\smallskip
(a) If $A_1 \otimes \dots \otimes A_r$ is split over $K$, then
$A_r$ can be recovered from $A_1, \dots, A_{r-1}$
as the unique central simple $K$-algebra of degree $p^{a_r}$
which is Brauer-equivalent to
\[ (A_1 \otimes \dots \otimes A_{r-1})^{\rm op} \, . \]
(Here $B^{\rm op}$ denotes the opposite algebra of $B$.) 
In other words, the morphism of functors  
\begin{equation} \label{e.functors}
\Func \to
H^1(-, \PGL_{p^{a_1}}) \times \dots \times  H^1(-, \PGL_{p^{a_{r-1}}}) 
\end{equation}
given by $(A_1, \dots, A_{r-1}, A_r) \to (A_1, \dots, A_{r-1})$
is injective. We claim that if $\ds a_r \geqslant a_1 + \dots + a_{r-1}$ (which
is our assumption in part (a)), then this morphism if also surjective.
Indeed, 
\[ \deg(A_1 \otimes \dots \otimes A_{r-1}) = p^{a_1 + \dots + a_{r-1}} \]
for any choice of central simple $K$-algebras $A_1, \dots, A_{r-1}$ 
such that $\deg(A_i) = p^{a_i}$. Hence, for any such choice there exists
a central simple algebra of degree $p^{a_r}$ which is Brauer-equivalent to 
$(A_1 \otimes \dots \otimes A_{r-1})^{\rm op}$. This proves the claim. 

We conclude that if $\ds a_r \geqslant a_1 + \dots + a_{r-1}$ 
then~\eqref{e.functors} is an isomorphism and thus
$\ed(\widetilde{G}) = \ed(G) = \ed(\Func) = \ed(\PGL_{p^{a_1}} \times
\dots \times \PGL_{p^{a_{r-1}}})$ and 
\[ \ed_p(\widetilde{G}) = \ed_p(G) = \ed_p(\Func) = \ed_p(\PGL_{p^{a_1}} \times
\dots \times \PGL_{p^{a_{r-1}}}) \, . \]
The inequality $\ed(\Func) 
\leqslant p^{2a_1} + \dots + p^{2a_{r-1}}$ now
follows from~\eqref{e.product}.

\smallskip
(b) Now suppose $\ds a_r < a_1 + \dots + a_{r-1}$.  In this case 
$\Code(C)$ has a minimal basis consisting of the single element
$(1, \dots, 1) \in (\bbZ/p^{a_1} \bbZ) \times \dots \times 
(\bbZ/p^{a_r} \bbZ)$. Moreover, 
$p^{a_r} \leqslant \dfrac{1}{2} p^{a_1} \dots p^{a_{r-1}}$ 
and consequently, Theorem~\ref{thm.main2} applies.
It tells us that if the $r$-tuple
$(p^{a_1}, \dots, p^{a_r})$ is not of the form 
$(2,2,2,2)$, $(3,3,3)$ or $(2, 2^a, 2^a)$, then
\[ \ds \ed(\Func) = \ed_p(\Func) = 
\ed(\widetilde{G}) = \ed_p(\widetilde{G}) = 
\ed(G) = \ed_p(G) = 
p^{a_1 + \dots + a_r} - \sum_{i=1}^rp^{2a_i} +r-1 \, , \]
as claimed.

\smallskip
(c) In the case, where $(p^{a_1}, \dots, p^{a_r}) = (2, 2, 2)$, 
$\Func(K)$ is the set of isomorphism classes of
triples $(A_1, A_2, A_3)$ of quaternion
$K$-algebras, such that $A_1 \otimes A_2 \otimes A_3$ is split over $K$.
We will show that (i) $\ed(\Func) \leqslant 3$ and 
(ii) $\ed_2(\Func) \geqslant 3$.

\smallskip
To prove (i), recall that
by a theorem of Albert \cite[Theorem III.4.8]{L05}, 
the condition that $A_1 \otimes A_2 \otimes A_3$ is split 
over $K$ implies that $A_1$ and $A_2$ are linked over $K$. 
That is, there exist $a, b, c \in K^*$ such that  
$A_1 \simeq (a,b)$ and $A_2 \simeq (a,c)$ over $K$. 
Hence, the triple $(A_1, A_2, A_3) 
\in \Func(K)$ descends to the triple $(B_1, B_2, B_3) \in \Func(K_0)$, where
$K_0 = k(a, b, c)$, $B_1 = (a,b)$, $B_2 = (a,c)$ 
and $B_3 = (a, bc)$ over $K_0$. 
Since $\trdeg(K_0/k) \leqslant 3$, assertion (i) follows.

To prove (ii), consider the morphism of functors
$f \colon  \Func          \to     H^1(-, \SO_4)$ given by
\[          f \colon (A_1, A_2, A_3) \mapsto  \alpha \, , \]
where $\alpha$ is a $4$-dimensional quadratic form 
such that 
\[ \alpha \oplus \HH \oplus \HH \cong N(A_1) \oplus (-N(A_2)) \, . \]
Here $\HH$ denotes the $2$-dimensional hyperbolic form
$\langle 1,-1 \rangle$, $N(A_1)$ denotes the norm 
form of $A_1$, and $-N(A_2)$ denotes the opposite norm form of $A_2$, i.e., 
the unique $4$-dimensional form such that 
$N(A_2) \oplus (-N(A_2))$ is hyperbolic. 
Since $N(A_1)$ and $N(A_2)$ are forms of discriminant $1$,
so is $\alpha$ (this will also be apparent from the explicit computations 
below). Thus we may view $\alpha$ as an element of the Galois cohomology set 
$H^1(K, \SO_4)$, which classifies 
$4$-dimensional quadratic forms of discriminant $1$ over $K$, 
up to isomorphism.  Note also that by the Witt Cancellation 
Theorem, $\alpha$ is unique up to isomorphism.  
We conclude that the morphism of functors $f$ is well defined. 

Equivalently, using the definition of the Albert form given 
in \cite[p.~69]{L05}, $\alpha$ is the unique 
$4$-dimensional quadratic form such that 
$\alpha \oplus \HH \cong q$, where $q$ is the $6$-dimensional
Albert form of $A_1$ and $A_2$. Here the Albert form of
$A_1$ and $A_2$ is isotropic, 
and hence, can be written as $\alpha \oplus \HH$, because $A_1$ and $A_2$ 
are linked; once again, see~\cite[Theorem III.4.8]{L05}.

Suppose $A_1=(a,b)$, $A_2=(a,c)$, and $A_3 = (a, bc)$,  
as above.  Then 
\[ N(A_1) = \langle\langle-a,-b\rangle\rangle = \langle1,-a,-b,ab\rangle \, , \]
and similarly $N(A_2) = \langle1,-a,-c,ac\rangle$;
see, e.g.,~\cite[Corollary III.2.2]{L05}.  Thus
$$N(A_1) \oplus (-N(A_2)) = \langle1,-1,-a,a,-b,c,ab,-ac\rangle \simeq
\langle-b,c,ab,-ac\rangle \oplus \HH \oplus \HH$$ 
and we obtain an explicit formula for $\alpha = f(A_1, A_2, A_3)$: 
$\alpha \cong \langle-b,c,ab,-ac\rangle$.

It is easy to see that any $4$-dimensional
quadratic form of discriminant $1$ over $K$ can be written as
$\langle-b,c,ab,-ac\rangle$ for some $a, b, c \in K^*$.
In other words, the morphism of functors
$f \colon  \Func          \to     H^1(-, \SO_4)$ is surjective.
Consequently,
\[ \ed_2(\Func) \geq \ed_2(H^1(- , \SO_4)) = \ed_2(\SO_4) \, ; \]  
see, e.g.,~\cite[Lemma 1.9]{BF03} or~\cite[Lemma 2.2]{icm}.
On the other hand, $\ed_2(\SO_4;2) =3$; 
see~\cite[Theorem 8.1(2) \& Remark 8.2]{RY00} or
\cite[Corollary 3.6(a)]{icm}. Thus
\[ \ed_2(\Func) \geqslant \ed_2(\SO_4) =  3 \, . \]  
This completes the proof of (ii) and thus of part (c) and of
Theorem~\ref{thm.main3}.
\qed

\bigskip
\noindent
{\sc Acknowledgements.}
The authors are grateful to the anonymous referee for 
a careful reading of our paper and numerous helpful comments.

\bigskip

\section*{Appendix: Galois Cohomology of central quotients of products of
general linear groups} 
\renewcommand{\thesection}{A}
\label{sect.appendix}


\begin{center}
by {\large Athena Nguyen} 
\footnote{This appendix is based on a portion of the author's 
Master's thesis completed at the University of British Columbia. 
The author gratefully acknowledges the financial 
support from the University of British Columbia and
the Natural Sciences and Engineering Research Council 
of Canada.}
\end{center}

\bigskip

In this appendix we will study the Galois cohomology of algebraic groups 
of the form 
\[
G := \Gamma/C,  
\]
where 
$\Gamma := \GL_{n_1} \times \dots \times \GL_{n_r}$ and
$C \subset Z(\Gamma) = \bbG_m^r$ is a central subgroup.
Here $n_1, \dots, n_r \geqslant 1$ are integers, not necessarily prime powers.
Let $\overline{G} := G/Z(G) = \PGL_{n_1} \times \dots \times \PGL_{n_r}
= \Gamma/Z(\Gamma)$.  
Recall that for any field $K/k$,
$H^1(K, \PGL_n)$ is naturally identified with the set 
of isomorphism classes of central simple $K$-algebras of degree $n$, and
\[ H^1(K, \overline{G}) = H^1(K, \PGL_{n_1}) \times \dots \times
H^1(K, \PGL_{n_r}) \]
with the set of $r$-tuples $(A_1, \dots, A_r)$ 
of central simple $K$-algebras such that $\deg(A_i) = n_i$. 
Denote by $\partial^i_K$ 
the coboundary map $H^1(K, \PGL_{n_i})\rightarrow H^2(K, \bbG_m)$ 
induced by the short exact sequence
\begin{align*}
1\rightarrow \bbG_m \rightarrow \GL_{n_i} \rightarrow \PGL_{n_i} \rightarrow 1.
\end{align*}
This map sends a central simple algebra $A_i$ to its Brauer class $[A_i]$
in $H^2(K, \bbG_m) = \Br(K)$.

Of particular interest to us will be 
\[ X(\bbG_m^r/C) = \{ (m_1, \dots, m_r) \in \bbZ^r \;  
| \; \tau_1^{m_1} \dots \tau_r^{m_r} = 1  \; \; 
\forall (\tau_1, \dots, \tau_r) \in \bbG_m^r \} \, , \]
as in~\eqref{e.integer-code}. 
We are now ready to state the main result of this appendix.

\begin{thm}\label{thm.appendix}
Let $\pi \colon G \to \overline{G} := \PGL_{n_1} \times \dots \times \PGL_{n_r}$
be the natural projection and  
$\pi_* \colon H^1(K,G) \rightarrow H^1(K,\overline{G})$ 
be the induced map in cohomology. Here $K/k$ is a field extension.
Then

\smallskip
(a) $\pi_* \colon H^1(K,G) \rightarrow H^1(K,\overline{G})$ 
is injective for every field $K/k$. 

\smallskip
(b) $\pi_*$  identifies $H^1(K,G)$ with the set of isomorphism classes  
of $r$-tuples $(A_1,\ldots, A_r)$ of central simple $K$-algebras 
such that $\deg(A_i) = n_i$ and
$A_1^{\otimes m_1} \otimes \cdots \otimes A_r^{\otimes m_r}$
is split over $K$ for every $(m_1,\ldots,m_r) \in X(\bbG_m^r/C)$.
\end{thm}

\begin{proof}
Throughout, we will identify $H^2(K,\mb{G}_m^r)$ with $H^2(K,\mb{G}_m)^r$ 
and $X(\mb{G}_m^r)$ with $\bbZ^n$. A character 
$x =(m_1,\ldots, m_r) \in \bbZ^n$, i.e., a character 
$x \colon \bbG_m^r \to \bbG_m$
given by $(\tau_1, \dots, \tau_r) \to \tau_1^{m_1} \dots \tau_r^{m_r}$,
induces a map 
$x_\ast \colon H^2(K,\mb{G}_m)^r \rightarrow H^2(K,\mb{G}_m)$
in cohomology given by
\begin{equation} \label{e.x_*}
x_*(\alpha_1,\ldots, \alpha_r) = \alpha_1^{m_1}\cdot 
\ldots \cdot \alpha_r^{m_r} \, .
\end{equation}

 Let us now consider the diagram
\[
\xymatrix{
&1 \ar[r] & \bbG_m^r \ar[r] \ar[d]_\eta & \Gamma \ar[r]\ar[d] 
&\prod\limits_{i=1}^r \PGL_{n_i} \ar[r]\ar@{=}[d] &1 \\
& 1 \ar[r] & \bbG_m^r/C \ar[r] & G \ar[r]^(.3){\pi} & 
\prod\limits_{i=1}^r \PGL_{n_i} \ar[r] &1. }
\]
Since $H^1(K, \bbG_m^r/C) = \{ 1 \}$ 
by Hilbert's Theorem 90, we obtain the following 
diagram in cohomology with exact rows:
\[
\xymatrix{
  &  & H^1(K, \prod\limits_{i=1}^r \PGL_{n_i}) 
\ar[rr]^(.55){(\partial^1_K,\ldots, \partial_K^r)} \ar@{=}[d] & &
H^2(K, \bbG_m^r) \ar[d]^{\eta_\ast} \\
  0 \ar[r] & H^1(K,G) \ar[r]^(.4){\pi_\ast} & 
H^1(K, \prod\limits_{i=1}^r \PGL_{n_i}) \ar[rr]^{\partial_K} & 
& H^2(K, \bbG_m^r/C)}
\]

(a) It follows from~\cite[I.5, Proposition 42]{S97} that $\pi_\ast$ 
is injective. 

(b) Thus, $\pi_*$ identifies $H^1(K,G)$ with the set of 
$r$-tuples $(A_1,\ldots, A_r)$, where $A_i \in H^1(K, \PGL_{n_i})$ 
is a central simple algebra of degree $n_i$,
and $(\partial_K^1 (A_1),\ldots, \partial_K^r(A_r)) \in 
\Ker (\eta_\ast)$. Recall that $\partial_K^i$ sends 
a central simple algebra $A_i$ to its Brauer class 
$[A_i] \in H^2(K, \bbG_m)$. 

Consider an $r$-tuple $\alpha := ([A_1],\ldots, [A_r]) \in H^2(K, \bbG_m^r)$.
Since $\bbG_m^r/C$ is diagonalizable, $\eta_\ast(\alpha)=0$ 
if and only if $x_\ast(\eta_\ast(\alpha))=0$ for all 
$x \in X(\bbG_m^r/C)$. 
If $x = (m_1, \dots, m_r) \in X(\bbG_m^r/C)$, then 
$x_* \circ \eta_* = (m_1, ... m_r) \in X(G_m^r)$.
By~\eqref{e.x_*},
$x_\ast(\eta_\ast(\alpha))= [A_1^{\otimes m_1}
\otimes \cdots \otimes A_r^{\otimes m_r}]$, and
part (b) follows. 
\end{proof}

\begin{cor} \label{cor1.appendix}
Let $\Gamma := \GL_{n_1} \times \dots \times \GL_{n_r}$,
$C_1, C_2$ be $k$-subgroups of $Z(\Gamma) = \bbG_m^r$,
$G_1 = \Gamma/C_1$ and $G_2 = \Gamma/C_2$. Denote the central subgroup 
$\mu_{n_1} \times \dots \times \mu_{n_r}$ of $\Gamma$ by $\mu$.

If $C_1 \cap \mu = C_2 \cap \mu$ then the
Galois cohomology functors $H^1(-, G_1)$ and 
$H^1(-, G_2)$ are isomorphic.
\end{cor}

\begin{proof} By Theorem~\ref{thm.appendix}, $H^1(K, G_i)$
is naturally identified with the set of $r$-tuples $(A_1, \dots, A_r)$
of central simple algebras such that $\deg(A_i) = n_i$
and 
\[ \text{$A_1^{\otimes m_1} \otimes \dots \otimes A_r^{\otimes m_r}$ 
 is split over $K$ for 
every $(m_1, \dots, m_r) \in X(\bbG_m/C_i)$.} \]
Note that since $A_i^{\otimes n_i}$ is split for every $i$,
this condition depends only on the image of $(m_1, \dots, m_r)$
under the natural projection 
\[ \pi \colon X(\bbG_m^r) = \bbZ^r \to (\bbZ/ n_1 \bbZ) \times \dots \times
(\bbZ/ n_r \bbZ) = X(\mu) \, .\]

Our assumption that $C_1 \cap \mu = C_2 \cap \mu$ is equivalent to
$X(\bbG_m^r/C_1)$ and   $X(\bbG_m^r/C_2)$ having 
the same image under $\pi$, and the corollary follows.
\end{proof}

In order to state the second corollary of Theorem~\ref{thm.appendix}, 
we will need the following definition. By a code we shall mean a subgroup
of $X(\mu) = (\bbZ/ n_1 \bbZ) \times \dots \times
(\bbZ/ n_r \bbZ)$. Given a subgroup $C \subset \mu$, we define
the code $\Code(C) := X(\mu/C)$, as in~\eqref{e.C}. 

We will say that two codes are called {\em equivalent} if one 
can be obtained from the other by repeatedly performing 
the following elementary operations:

\smallskip
(1) Permuting entries $i$ and $j$ in every vector of the code, 
for any $i,j$ with $n_i = n_j$.

\smallskip
(2) Multiplying the $i^{th}$ entry in every vector of the 
code by an integer $c$ prime to $n_i$.

\begin{cor}\label{cor2.appendix}
Suppose $C_1$ and $C_2$ are subgroups of 
$\mu := \mu_{n_1} \times \dots \times \mu_{n_r}$, 
$G_1 = \Gamma/C_1$ and $G_2:= \Gamma/C_2$.
If $\Code(C_1)$ and $\Code(C_2)$ are equivalent, then

\smallskip
(a) the Galois cohomology functors $H^1(-,G_1)$ and $H^1(-, G_2)$
are isomorphic, and 

\smallskip
(b) in particular, $\ed(G_1) = \ed(G_2)$ and $\ed_p(G_1) = \ed_p(G_2)$
for every prime $p$.
\end{cor}

\begin{proof} 
(a) It suffices to show that $H^1(-, G_1)$ and $H^1(-,G_2)$
are isomorphic if $C_2$ is obtained from $C_1$ by an
elementary operation.

(1) Suppose $n_i = n_j$ for some $i, j = 1, \dots, r$, and
$\Code(C_2)$ is obtained from $\Code(C_1)$ by permuting 
entries $i$ and $j$ in every vector.
In this case $C_2 = \alpha(C_1)$, where $\alpha$ is the automorphism
of $\Gamma = \GL_{n_1} \times \dots \times \GL_{n_r}$ which swaps
the $i$th and the $j$th components. Then $\alpha$ induces an isomorphism
between $G_1 = \Gamma/C_1$ and $G_2 = \Gamma/C_2$, and thus 
an isomorphism between $H^1(-, G_1)$ and $H^1( -, G_2)$. 

(2) Now suppose that $\Code(C_1)$ is 
obtained from $\Code(C_2)$ by multiplying 
the $i^{th}$ entry in every vector by some $c \in (\bbZ/ n_i \bbZ)^*$. 
The description of $H^1(K, G/\mu)$ given 
by Theorem~\ref{thm.appendix} now tells us that
\begin{align*}
H^1(K,G_1) &\rightarrow H^1(K,G_2)\\
(A_1,\ldots, A_r)&\mapsto (A_1,\ldots, A_{i-1}, [A_i^{\otimes c}]_{n_i}, 
A_{i+1}, \ldots, A_r)
\end{align*}
is an isomorphism. Here, by $[A_i^{\otimes c}]_{n_i}$ we mean the unique
central simple $K$-algebra of degree $n_i$ which 
is Brauer equivalent to $A_i^{\otimes c}$. 

\smallskip
(b) follows from (a), because $\ed(G)$ and $\ed_p(G)$ are defined 
entirely in terms of the Galois cohomology functor $H^1(-, G)$.
\end{proof}

\end{document}